\documentclass[10pt,portrait,reqno,12pt]{amsart}%
\usepackage{amssymb,amsthm,amsmath, amsfonts, float}
\usepackage{graphicx}
\usepackage{multirow}
\usepackage{subfig}
\usepackage{amsmath}
\usepackage{amsfonts}
\usepackage{lscape}
\usepackage{amssymb}%
\providecommand{\U}[1]{\protect\rule{.1in}{.1in}}
\newtheorem{theorem}{Theorem}
\theoremstyle{plain}

\newtheorem{corollary}{Corollary}

\numberwithin{equation}{section}
\begin{document}
\title[Canal Hypersurfaces]{Geometric Characterizations of Canal Hypersurfaces in Euclidean Spaces}
\author[A. Kazan, M. Alt\i n and D.W. Yoon]{\bfseries Ahmet Kazan$^{1\ast}$, Mustafa Alt\i n$^{2}$ and Dae Won Yoon$^{3}$}
\address{$^{1}$Department of Computer Technologies, Do\u{g}an\c{s}ehir Vahap
K\"{u}\c{c}\"{u}k Vocational School, Malatya Turgut \"{O}zal University,
Malatya, Turkey\\
$^{2}$Technical Sciences Vocational School, Bing\"{o}l University, Bing\"{o}l,
Turkey \\
$^{3}$Department of Mathematics Education and RINS, Gyeongsang National
University, Jinju 52828, Republic of Korea\\
 \\
$^{\ast}$Corresponding author: ahmet.kazan@ozal.edu.tr}

\begin{abstract}
In the present paper, firstly we obtain the general expression of canal
hypersurfaces in Euclidean $n$-space and deal with canal hypersurfaces in
Euclidean 4-space $E^{4}$. We compute Gauss map, Gaussian curvature and mean
curvature of canal hypersurfaces in $E^{4}$ and obtain an important relation
between the mean and Gaussian curvatures as $3H\rho=K\rho^{3}-2$. We prove
that, the flat canal hypersurfaces in Euclidean 4-space are only circular
hypercylinders or circular hypercones and minimal canal hypersurfaces are only
generalized catenoids. Also, we state the expression of tubular hypersurfaces
in Euclidean spaces and give some results about Weingarten tubular
hypersurfaces in $E^{4}$.

\end{abstract}
\maketitle

%=====Contents=======

\section{Introduction}

A canal surface is formed by the envelope of the spheres whose centers lie on
a curve and radius vary depending on this curve \cite{Gray}. In this sense,
let $\Lambda:=\alpha(u)=(a(u),b(u),c(u))$ be a regular space curve and $r(u)$
be a $C^{1}$-function with $r>0$ and $\left\vert \dot{r}\right\vert
<\left\Vert \dot{\alpha}\right\Vert $. The envelope of the one parameter
family of spheres%
\begin{equation}
(x-\alpha(u))^{2}-r(u)^{2}=0 \label{canal}%
\end{equation}
is called a \textit{canal surface} and $\Lambda$ its \textit{directrix }in
Euclidean 3-space\textit{.} Also, the parametric representation of canal
surfaces can be given by%
\begin{equation}
\mathbf{x}=\mathbf{x}(u,v):=\alpha(u)-\frac{r(u)\dot{r}(u)}{\left\Vert
\dot{\alpha}(u)\right\Vert ^{2}}\dot{\alpha}(u)+\frac{r(u)\sqrt{\left\Vert
\dot{\alpha}(u)\right\Vert ^{2}-\dot{r}(u)^{2}}}{\left\Vert \dot{\alpha
}(u)\right\Vert }(e_{1}(u)\cos(v)+e_{2}(u)\sin(v)), \label{canal2}%
\end{equation}
where $\{e_{1},e_{2}\}$ is an orthonormal base orthogonal to tangent vector
$\dot{\alpha}$. In case of a constant radius function, the envelope is called
\textit{tubular} or \textit{pipe surface} (see \cite{Hartman}). Canal surfaces
(especially \textit{tubular} surfaces) have been applied to many fields, such
as the solid and the surface modeling for CAD/CAM, construction of blending
surfaces, shape re-construction and so on.

In this context, canal and tubular (hyper)surfaces have been studied by many
mathematicians in different spaces. For instance, the notion of special
conformally flat spaces which generalizes that of subprojective spaces has
been introduced in \cite{Chen} and the authors have proved that every canal
hypersurface of a Euclidean space is a special conformally flat space and it
is a subprojective space if and only if it is a surface of revolution. In
\cite{Izumiya}, a relationship between the caustics of a submanifold of
general dimension and of a canal hypersurface of the submanifold in Euclidean
space has been investigated and as a consequence, it has been seen that these
caustics are same. Analytic and algebraic properties of canal surfaces have
been studied in \cite{Xu}. In \cite{Aslan}, the authors have shown that canal
surfaces and tube surfaces can be obtained by the quaternion product and by
the matrix representation and also in \cite{Peter}, it is shown that any canal
surface to a rational spine curve and a rational radius function possesses
rational parametrizations. The principal curvatures and principal curvature
lines on canal surfaces have been determined in \cite{Garcia} and by means of
a connection of the differential equations for these curvature lines and real
Riccati equations, it has been established that canal surfaces have at most
two isolated periodic principal lines. Some interesting and important
relations about the Gaussian curvature, the mean curvature and the second
Gaussian curvature have been found and based on these relations, some canal
surfaces have been characterized in \cite{Kim}. Classification of cyclic
surfaces which is formed by movement of a circle of variable or constant
radius under any law in a three dimensional space and geometrical research of
canal surfaces have been given in \cite{Krivos}.

Furthermore, for different studies of canal and tubular surfaces in different
spaces such as Minkowskian, Galilean and pseudo-Galilean, we refer to
\cite{Karacan}, \cite{Karacan2}, \cite{Karacan3}, \cite{Kucuk}, \cite{Maekawa}%
, \cite{Ro}, \cite{Ucum} and etc.

In the second section of this paper, we obtain the general expression of canal
hypersurfaces in Euclidean $n$-space. In the third section, after recalling
some basic notions about hypersurfaces and stating the expression of canal
hypersurfaces in 4-dimensional Euclidean space, we obtain the Gaussian
curvature and the mean curvature of canal hypersurfaces in $E^{4}$ and give an
important relation between these curvatures. Moreover, we study on tubular
hypersurfaces in this section.

\section{Expression of Canal Hypersurfaces in Euclidean $n$-Space}

Let a center curve $\alpha:I\subseteq%
%TCIMACRO{\U{211d} }%
%BeginExpansion
\mathbb{R}
%EndExpansion
\rightarrow E^{n}$ be a curve with non-zero curvature and arc-length
parametrization. Then, the parametrization of the envelope of hypersphere
defining the canal hypersurface $X$ in $E^{n}$ can be given by%
\begin{equation}
X(\upsilon_{1},\upsilon_{2},\upsilon_{3},...,\upsilon_{n-1})-\alpha
(\upsilon_{1})=\sum_{i=1}^{n}a_{i}(\upsilon_{1},\upsilon_{2},\upsilon
_{3},...,\upsilon_{n-1})F_{i}(\upsilon_{1}),\text{ }i\in\left\{
1,2,3,...,n\right\}  ,\label{1}%
\end{equation}
where $F_{i}(\upsilon_{1})$ are Frenet vectors of $\alpha(\upsilon_{1})~$and
$a_{i}$ are differentiable functions of $\upsilon_{1},\upsilon_{2}%
,\upsilon_{3},...,\upsilon_{n-1}$ on the interval $I$. Furthermore, since
$X(\upsilon_{1},\upsilon_{2},\upsilon_{3},...,\upsilon_{n-1})$ lies on the
hypersphere, we have%
\begin{equation}
\left\langle X(\upsilon_{1},\upsilon_{2},\upsilon_{3},...,\upsilon
_{n-1})-\alpha(\upsilon_{1}),X(\upsilon_{1},\upsilon_{2},\upsilon
_{3},...,\upsilon_{n-1})-\alpha(\upsilon_{1})\right\rangle =\rho^{2}%
(\upsilon_{1})\label{2}%
\end{equation}
which leads to from (\ref{1}) that%
\begin{equation}
\sum_{i=1}^{n}\left(  a_{i}(\upsilon_{1},\upsilon_{2},\upsilon_{3}%
,...,\upsilon_{n-1})\right)  ^{2}=\rho^{2}(\upsilon_{1})\label{3}%
\end{equation}
and%
\begin{equation}
\sum_{i=1}^{n}a_{i}(\upsilon_{1},\upsilon_{2},\upsilon_{3},...,\upsilon
_{n-1})(a_{i}(\upsilon_{1},\upsilon_{2},\upsilon_{3},...,\upsilon
_{n-1}))_{\upsilon_{1}}=\rho(\upsilon_{1})\rho^{\prime}(\upsilon
_{1}).\label{4}%
\end{equation}
Here, $\rho(\upsilon_{1})$ is the radius function of hypersurface $X$ and we
note that, throughout this study, we state $\rho^{\prime}(\upsilon_{1}%
)=\frac{d\rho(\upsilon_{1})}{dv_{1}}$, $(a_{i}(\upsilon_{1},\upsilon
_{2},\upsilon_{3},...,\upsilon_{n-1}))_{v_{i}}=\frac{\partial a_{i}%
(\upsilon_{1},\upsilon_{2},\upsilon_{3},...,\upsilon_{n-1})}{\partial v_{i}}$,
$(X_{i}(\upsilon_{1},\upsilon_{2},\upsilon_{3},...,\upsilon_{n-1}))_{v_{i}%
}=\frac{\partial X_{i}(\upsilon_{1},\upsilon_{2},\upsilon_{3},...,\upsilon
_{n-1})}{\partial v_{i}},$ $i\in\left\{  1,2,...,n-1\right\}  $.

We know that \cite{Cheng}, the Frenet $n$-frame $F_{i}(\upsilon_{1})$ of the
curve $\alpha(\upsilon_{1})$ satisfy the relations%

\begin{equation}
\left.
\begin{array}
[c]{l}%
F_{1}^{\prime}(\upsilon_{1})=k_{1}(\upsilon_{1})F_{2}(\upsilon_{1}),\\
\\
F_{i}^{\prime}(\upsilon_{1})=-k_{i-1}(\upsilon_{1})F_{i-1}(\upsilon_{1}%
)+k_{i}(\upsilon_{1})F_{i+1}(\upsilon_{1}),\text{ \ }i\in\left\{
2,3,...,n-1\right\}  ,\\
\\
F_{n}^{\prime}(\upsilon_{1})=-k_{n-1}(\upsilon_{1})F_{n-1}(\upsilon_{1}),
\end{array}
\right\}  \label{frenet}%
\end{equation}
where $k_{i}$ are the $i$-th curvatures of the curve $\alpha(\upsilon_{1}).$

So, differentiating (\ref{1}) with respect to $\upsilon_{1}$ and using the
Frenet formula (\ref{frenet}), we get%
\begin{align}
(X(\upsilon_{1},\upsilon_{2},\upsilon_{3},...,\upsilon_{n-1}))_{\upsilon_{1}}
&  =F_{1}(\upsilon_{1})+\sum_{i=1}^{n}(a_{i}(\upsilon_{1},\upsilon
_{2},\upsilon_{3},...,\upsilon_{n-1}))_{\upsilon_{1}}F_{i}(\upsilon
_{1})\nonumber\\
&  +a_{1}(\upsilon_{1},\upsilon_{2},\upsilon_{3},...,\upsilon_{n-1}%
)k_{1}(\upsilon_{1})F_{2}(\upsilon_{1})\nonumber\\
&  +\sum_{i=2}^{n-1}a_{i}(\upsilon_{1},\upsilon_{2},\upsilon_{3}%
,...,\upsilon_{n-1})(-k_{i-1}(\upsilon_{1})F_{i-1}(\upsilon_{1})+k_{i}%
(\upsilon_{1})F_{i+1}(\upsilon_{1}))\nonumber\\
&  +a_{n}(\upsilon_{1},\upsilon_{2},\upsilon_{3},...,\upsilon_{n-1}%
)(-k_{n-1}(\upsilon_{1})F_{n-1}(\upsilon_{1})). \label{5}%
\end{align}
Furthermore, $X(\upsilon_{1},\upsilon_{2},\upsilon_{3},...,\upsilon
_{n-1})-\alpha(\upsilon_{1})$ is a normal vector to the canal hypersurfaces,
which implies that%
\begin{equation}
\left\langle X(\upsilon_{1},\upsilon_{2},\upsilon_{3},...,\upsilon
_{n-1})-\alpha(\upsilon_{1}),(X(\upsilon_{1},\upsilon_{2},\upsilon
_{3},...,\upsilon_{n-1}))_{v_{i}}\right\rangle =0,\text{ }i\in\left\{
1,2,3,...,n-1\right\}  . \label{6}%
\end{equation}
Then, taking $i=1$ in (\ref{6}), from (\ref{1}), (\ref{3}), (\ref{4}) and
(\ref{5}) we obtain%
\begin{equation}
\left.
\begin{array}
[c]{l}%
a_{1}(\upsilon_{1},\upsilon_{2},\upsilon_{3},...,\upsilon_{n-1})=-\rho
(\upsilon_{1})\rho^{\prime}(\upsilon_{1}),\\%
%TCIMACRO{\dsum \limits_{i=2}^{n}}%
%BeginExpansion
{\displaystyle\sum\limits_{i=2}^{n}}
%EndExpansion
\left(  a_{i}(\upsilon_{1},\upsilon_{2},\upsilon_{3},...,\upsilon
_{n-1})\right)  ^{2}=\rho^{2}(\upsilon_{1})(1-(\rho^{\prime}(\upsilon
_{1}))^{2}).
\end{array}
\right\}  \label{7}%
\end{equation}
From (\ref{7}), let us take%
\begin{equation}
\left.
\begin{array}
[c]{l}%
a_{2}(\upsilon_{1},\upsilon_{2},\upsilon_{3},...,\upsilon_{n-1})=\pm
\rho(\upsilon_{1})\sqrt{1-(\rho^{\prime}(\upsilon_{1}))^{2}}\prod
\limits_{k=2}^{n-1}\cos(x_{k}),\\
a_{i}(\upsilon_{1},\upsilon_{2},\upsilon_{3},...,\upsilon_{n-1})=\pm
\rho(\upsilon_{1})\sqrt{1-(\rho^{\prime}(\upsilon_{1}))^{2}}\sin
(\upsilon_{n+1-i})\prod\limits_{k=n+2-i}^{n-1}\cos(x_{k}),\\
a_{n}(\upsilon_{1},\upsilon_{2},\upsilon_{3},...,\upsilon_{n-1})=\pm
\rho(\upsilon_{1})\sqrt{1-(\rho^{\prime}(\upsilon_{1}))^{2}}\sin
(\upsilon_{n-1}),
\end{array}
\right\}  \label{8}%
\end{equation}
where $i\in\left\{  3,4...,n-1\right\}  .$

So, from (\ref{1}) and (\ref{8}), we have

\begin{theorem}
The canal hypersurface in Euclidean $n$-space is expressed as%
\begin{align}
&  X(\upsilon_{1},\upsilon_{2},\upsilon_{3},...,\upsilon_{n-1})=\alpha
(\upsilon_{1})-\rho(\upsilon_{1})\rho^{\prime}(\upsilon_{1})F_{1}(\upsilon
_{1})\nonumber\\
&  \text{ \ \ \ \ \ \ \ \ \ \ \ \ \ \ \ \ \ \ \ \ }\pm\rho(\upsilon_{1}%
)\sqrt{1-(\rho^{\prime}(\upsilon_{1}))^{2}}\left[
\begin{array}
[c]{l}%
\left(  \prod\limits_{k=2}^{n-1}\cos(\upsilon_{k})\right)  F_{2}(\upsilon
_{1})\\
+%
%TCIMACRO{\dsum \limits_{i=3}^{n-1}}%
%BeginExpansion
{\displaystyle\sum\limits_{i=3}^{n-1}}
%EndExpansion
\left(  \sin(\upsilon_{n+1-i})\prod\limits_{k=n+2-i}^{n-1}\cos(\upsilon
_{k})\right)  F_{i}(\upsilon_{1})\\
+\sin(\upsilon_{n-1})F_{n}(\upsilon_{1})
\end{array}
\right]  . \label{n-canal}%
\end{align}

\end{theorem}

\section{Canal Hypersurfaces in $E^{4}$}

In this section, we study canal hypersurfaces in Euclidean 4-space $E^{4}$ by
giving their expressions with the aid of (\ref{n-canal}).

Since we will deal with canal hypersurface in $E^{4}$ and give some important
characterizations about them, let we recall some fundamental notions for
hypersurfaces in $E^{4}$.

If $\overrightarrow{u}=(u_{1},u_{2},u_{3},u_{4})$, $\overrightarrow{v}%
=(v_{1},v_{2},v_{3},v_{4})$ and $\overrightarrow{w}=(w_{1},w_{2},w_{3},w_{4})$
are three vectors in $E^{4}$, then the inner product and vector product are
defined by
\begin{equation}
\left\langle \overrightarrow{u},\overrightarrow{v}\right\rangle =u_{1}%
v_{1}+u_{2}v_{2}+u_{3}v_{3}+u_{4}v_{4} \label{yy1}%
\end{equation}
and%
\begin{equation}
\overrightarrow{u}\times\overrightarrow{v}\times\overrightarrow{w}=\det\left[
\begin{array}
[c]{cccc}%
e_{1} & e_{2} & e_{3} & e_{4}\\
u_{1} & u_{2} & u_{3} & u_{4}\\
v_{1} & v_{2} & v_{3} & v_{4}\\
w_{1} & w_{2} & w_{3} & w_{4}%
\end{array}
\right]  , \label{yy2}%
\end{equation}
respectively. Also, the norm of the vector $\overrightarrow{u}$ is $\left\Vert
\overrightarrow{u}\right\Vert =\sqrt{\left\langle \overrightarrow{u}%
,\overrightarrow{u}\right\rangle }.$

If
\begin{align}
\Psi:U\subset E^{3}  &  \longrightarrow E^{4}\label{yy3}\\
(\upsilon_{1},\upsilon_{2},\upsilon_{3})  &  \longrightarrow\Psi(\upsilon
_{1},\upsilon_{2},\upsilon_{3})=(\Psi_{1}(\upsilon_{1},\upsilon_{2}%
,\upsilon_{3}),\Psi_{2}(\upsilon_{1},\upsilon_{2},\upsilon_{3}),\Psi
_{3}(\upsilon_{1},\upsilon_{2},\upsilon_{3}),\Psi_{4}(\upsilon_{1}%
,\upsilon_{2},\upsilon_{3}))\nonumber
\end{align}
is a hypersurface in $E^{4}$, then the unit normal vector field, the matrix
forms of the first and second fundamental forms are%

\begin{equation}
N_{\Psi}=\frac{\Psi_{\upsilon_{1}}\times\Psi_{\upsilon_{2}}\times
\Psi_{\upsilon_{3}}}{\left\Vert \Psi_{\upsilon_{1}}\times\Psi_{\upsilon_{2}%
}\times\Psi_{\upsilon_{3}}\right\Vert }, \label{4y}%
\end{equation}%
\begin{equation}
\lbrack g_{ij}]=\left[
\begin{array}
[c]{ccc}%
g_{11} & g_{12} & g_{13}\\
g_{21} & g_{22} & g_{23}\\
g_{31} & g_{32} & g_{33}%
\end{array}
\right]  \label{5y}%
\end{equation}
and%
\begin{equation}
\lbrack h_{ij}]=\left[
\begin{array}
[c]{ccc}%
h_{11} & h_{12} & h_{13}\\
h_{21} & h_{22} & h_{23}\\
h_{31} & h_{32} & h_{33}%
\end{array}
\right]  , \label{6y}%
\end{equation}
respectively. Here $g_{ij}=\left\langle \Psi_{\upsilon_{i}},\Psi_{\upsilon
_{j}}\right\rangle ,$ $h_{ij}=\left\langle \Psi_{\upsilon_{i}\upsilon_{j}%
},N_{\Psi}\right\rangle ,$ $\Psi_{\upsilon_{i}}=\frac{\partial\Psi
(\upsilon_{1},\upsilon_{2},\upsilon_{3})}{\partial\upsilon_{i}},$
$\Psi_{\upsilon_{i}\upsilon_{j}}=\frac{\partial^{2}\Psi(\upsilon_{1}%
,\upsilon_{2},\upsilon_{3})}{\partial\upsilon_{i}\upsilon_{j}},$
$i,j\in\{1,2,3\}.$ Also, the shape operator of the hypersurface (\ref{yy3}) is%
\begin{equation}
S=[a_{ij}]=[g^{ij}][h_{ij}], \label{7y}%
\end{equation}
where $[g^{ij}]$ is the inverse matrix of $[g_{ij}]$.

With the aid of (\ref{4y})-(\ref{7y}), the Gaussian and mean curvatures of a
hypersurface in $E^{4}$ are given by%
\begin{equation}
K=\det(S)=\frac{\det[h_{ij}]}{\det[g_{ij}]} \label{yy4}%
\end{equation}
and%
\begin{equation}
3H=tr(S), \label{yy5}%
\end{equation}
respectively \cite{Guler}.

From (\ref{n-canal}), the canal hypersurface $\emph{C}$ in $E^{4}$ can be
written as%
\begin{equation}
\emph{C}(\upsilon_{1},\upsilon_{2},\upsilon_{3})=\alpha(\upsilon_{1})-\rho
\rho^{\prime}F_{1}(\upsilon_{1})\pm\rho\sqrt{1-\rho^{\prime}{}^{2}}\left[
\cos\upsilon_{2}\cos\upsilon_{3}F_{2}(\upsilon_{1})+\sin\upsilon_{2}%
\cos\upsilon_{3}F_{3}(\upsilon_{1})+\sin\upsilon_{3}F_{4}(\upsilon
_{1})\right]  , \label{4-canal}%
\end{equation}
where $\upsilon_{1}\in\left[  0,l\right]  $ and$~\upsilon_{2},\upsilon_{3}%
\in\left[  0,2\pi\right)  $. Also, from now on we state $\alpha=\alpha
(\upsilon_{1}),$ $\rho=\rho(\upsilon_{1}),$ $F_{i}=F_{i}(\upsilon_{1}),$
$i\in\{1,2,3,4\},$ $\rho^{\prime}=\frac{d\rho(\upsilon_{1})}{d\upsilon_{1}}$
and we will consider the "$\pm$" as "$+$".

Firstly, from (\ref{frenet}) and (\ref{4-canal}) the derivatives of the canal
hypersurface (\ref{4-canal}) are obtained as%
\begin{align}
\emph{C}_{\upsilon_{1}}  &  =\emph{C}_{\upsilon_{1}}^{\text{ }1}F_{1}%
+\emph{C}_{\upsilon_{1}}^{\text{ }2}F_{2}+\emph{C}_{\upsilon_{1}}^{\text{ }%
3}F_{3}+\emph{C}_{\upsilon_{1}}^{\text{ }4}F_{4},\label{9}\\
\emph{C}_{\upsilon_{2}}  &  =-\rho\sqrt{1-\rho^{\prime}{}^{2}}\sin\upsilon
_{2}\cos\upsilon_{3}F_{2}+\rho\sqrt{1-\rho^{\prime}{}^{2}}\cos\upsilon_{2}%
\cos\upsilon_{3}F_{3},\label{10}\\
\emph{C}_{\upsilon_{3}}  &  =-\rho\sqrt{1-\rho^{\prime}{}^{2}}\cos\upsilon
_{2}\sin\upsilon_{3}F_{2}-\rho\sqrt{1-\rho^{\prime}{}^{2}}\sin\upsilon_{2}%
\sin\upsilon_{3}F_{3}+\rho\sqrt{1-\rho^{\prime}{}^{2}}\cos\upsilon_{3}F_{4},
\label{11}%
\end{align}
where%
\[
\left.
\begin{array}
[c]{l}%
\emph{C}_{\upsilon_{1}}^{\text{ }1}=1-\rho^{\prime}{}^{2}-k_{1}\rho
\sqrt{1-\rho^{\prime}{}^{2}}\cos\upsilon_{2}\cos\upsilon_{3}-\rho\rho
^{\prime\prime},\\
\\
\emph{C}_{\upsilon_{1}}^{\text{ }2}=-k_{1}\rho\rho^{\prime}-k_{2}\rho
\sqrt{1-\rho^{\prime}{}^{2}}\sin\upsilon_{2}\cos\upsilon_{3}+\left(
\rho^{\prime}\sqrt{1-\rho^{\prime}{}^{2}}-\frac{\rho\rho^{\prime}\rho
^{\prime\prime}}{\sqrt{1-\rho^{\prime}{}^{2}}}\right)  \cos\upsilon_{2}%
\cos\upsilon_{3},\\
\emph{C}_{\upsilon_{1}}^{\text{ }3}=\rho\sqrt{1-\rho^{\prime}{}^{2}}\left(
k_{2}\cos\upsilon_{2}\cos\upsilon_{3}-k_{3}\sin\upsilon_{3}\right)  +\left(
\rho^{\prime}\sqrt{1-\rho^{\prime}{}^{2}}-\frac{\rho\rho^{\prime}\rho
^{\prime\prime}}{\sqrt{1-\rho^{\prime}{}^{2}}}\right)  \sin\upsilon_{2}%
\cos\upsilon_{3},\\
\emph{C}_{\upsilon_{1}}^{\text{ }4}=k_{3}\rho\sqrt{1-\rho^{\prime}{}^{2}}%
\sin\upsilon_{2}\cos\upsilon_{3}+\left(  \rho^{\prime}\sqrt{1-\rho^{\prime}%
{}^{2}}-\frac{\rho\rho^{\prime}\rho^{\prime\prime}}{\sqrt{1-\rho^{\prime}%
{}^{2}}}\right)  \sin\upsilon_{3}.
\end{array}
\right\}
\]
From (\ref{4y}) and (\ref{9})-(\ref{11}), the unit normal vector field of
$\emph{C}$ in $E^{4}$ is%
\begin{equation}
N=-\rho^{\prime}F_{1}+\sqrt{1-\rho^{\prime}{}^{2}}\cos\upsilon_{2}\cos
\upsilon_{3}F_{2}+\sqrt{1-\rho^{\prime}{}^{2}}\sin\upsilon_{2}\cos\upsilon
_{3}F_{3}+\sqrt{1-\rho^{\prime}{}^{2}}\sin\upsilon_{3}F_{4}. \label{12}%
\end{equation}
Also, the coefficients of the first fundamental form are given by%
\begin{equation}
\left.
\begin{array}
[c]{l}%
g_{11}=\frac{1}{1-\rho^{\prime2}}\left(
\begin{array}
[c]{l}%
\left(  1-\rho^{\prime}{}^{2}\right)  Q{}^{2}+\left(
\begin{array}
[c]{l}%
k_{2}\rho\left(  1-\rho^{\prime}{}^{2}\right)  \sin\upsilon_{2}\cos
\upsilon_{3}+k_{1}\rho\rho^{\prime}\sqrt{1-\rho^{\prime2}}\\
+\rho^{\prime}\left(  \rho^{\prime2}+\rho\rho^{\prime\prime}-1\right)
\cos\upsilon_{2}\cos\upsilon_{3}%
\end{array}
\right)  ^{2}{}\\
+\left(
\begin{array}
[c]{l}%
-k_{2}\rho\left(  1-\rho^{\prime}{}^{2}\right)  \cos\upsilon_{2}\cos
\upsilon_{3}+k_{3}\rho\left(  1-\rho^{\prime}{}^{2}\right)  \sin\upsilon_{3}\\
+\rho^{\prime}\left(  \rho^{\prime2}+\rho\rho^{\prime\prime}-1\right)
\sin\upsilon_{2}\cos\upsilon_{3}%
\end{array}
\right)  ^{2}\\
+\left(  \rho^{\prime}\left(  \rho^{\prime2}+\rho\rho^{\prime\prime}-1\right)
\sin\upsilon_{3}-k_{3}\rho\left(  1-\rho^{\prime}{}^{2}\right)  \sin
\upsilon_{2}\cos\upsilon_{3}\right)  {}^{2}%
\end{array}
\right)  ,\\
\\
g_{12}=g_{21}=\rho^{2}\left(  k_{1}\rho^{\prime}\sqrt{1-\rho^{\prime}{}^{2}%
}\sin\upsilon_{2}+k_{2}(1-\rho^{\prime}{}^{2})\cos\upsilon_{3}-k_{3}%
(1-\rho^{\prime}{}^{2})\cos\upsilon_{2}\sin\upsilon_{3}\right)  \cos
\upsilon_{3},\\
\\
g_{13}=g_{31}=\rho^{2}\left(  k_{1}\rho^{\prime}\sqrt{1-\rho^{\prime}{}^{2}%
}\cos\upsilon_{2}\sin\upsilon_{3}+k_{3}(1-\rho^{\prime}{}^{2})\sin\upsilon
_{2}\right)  ,\\
\\
g_{22}=\rho^{2}(1-\rho^{\prime}{}^{2})\cos^{2}\upsilon_{3},\\
\\
g_{23}=g_{32}=0,\\
\\
g_{33}=\rho^{2}(1-\rho^{\prime}{}^{2}),
\end{array}
\right\}  \label{13}%
\end{equation}
where $Q=\rho\left(  k_{1}\sqrt{1-\rho^{\prime}{}^{2}}\cos\upsilon_{2}%
\cos\upsilon_{3}+\rho^{\prime\prime}\right)  -1+\rho^{\prime}{}^{2}$ and it
follows that%
\begin{equation}
\det[g_{ij}]=\rho^{4}(1-\rho^{\prime}{}^{2})Q^{2}\cos^{2}\upsilon_{3}.
\label{14}%
\end{equation}

Now, for obtaining the coefficients of the second fundamental form, let we
give the second derivatives $\emph{C}_{\upsilon_{i}\upsilon_{j}}%
=\frac{\partial^{2}\emph{C}}{\partial\upsilon_{i}\upsilon_{j}}$ of the canal
hypersurface (\ref{4-canal}):%
\begin{equation}
\emph{C}_{\upsilon_{1}\upsilon_{1}}=\emph{C}_{\upsilon_{1}\upsilon_{1}%
}^{\text{ }1}F_{1}+\emph{C}_{\upsilon_{1}\upsilon_{1}}^{\text{ }2}%
F_{2}+\emph{C}_{\upsilon_{1}\upsilon_{1}}^{\text{ }3}F_{3}+\emph{C}%
_{\upsilon_{1}\upsilon_{1}}^{\text{ }4}F_{4}, \label{15}%
\end{equation}%
\begin{align}
\emph{C}_{\upsilon_{1}\upsilon_{2}}  &  =\emph{C}_{\upsilon_{2}\upsilon_{1}%
}\nonumber\\
&  =k_{1}\rho\sqrt{1-\rho^{\prime}{}^{2}}\sin\upsilon_{2}\cos\upsilon_{3}%
F_{1}\nonumber\\
&  +\left(  \frac{\cos\upsilon_{3}}{\sqrt{1-\rho^{\prime}{}^{2}}}\left(
-\rho^{\prime}(1-\rho^{\prime}{}^{2})\sin\upsilon_{2}+\rho(-k_{2}%
(1-\rho^{\prime}{}^{2})\cos\upsilon_{2}+\rho^{\prime}\rho^{\prime\prime}%
\sin\upsilon_{2})\right)  \right)  F_{2}\nonumber\\
&  +\left(  \frac{\cos\upsilon_{3}}{\sqrt{1-\rho^{\prime}{}^{2}}}\left(
\rho^{\prime}(1-\rho^{\prime}{}^{2})\cos\upsilon_{2}+\rho(-k_{2}%
(1-\rho^{\prime}{}^{2})\sin\upsilon_{2}-\rho^{\prime}\rho^{\prime\prime}%
\cos\upsilon_{2})\right)  \right)  F_{3}\nonumber\\
&  +k_{3}\rho\sqrt{1-\rho^{\prime}{}^{2}}\cos\upsilon_{2}\cos\upsilon_{3}%
F_{4}, \label{16}%
\end{align}%
\begin{align}
\emph{C}_{\upsilon_{1}\upsilon_{3}}  &  =\emph{C}_{\upsilon_{3}\upsilon_{1}%
}\nonumber\\
&  =k_{1}\rho\sqrt{1-\rho^{\prime}{}^{2}}\cos\upsilon_{2}\sin\upsilon_{3}%
F_{1}\nonumber\\
&  +\left(  \frac{\sin\upsilon_{3}}{\sqrt{1-\rho^{\prime}{}^{2}}}\left(
-\rho^{\prime}(1-\rho^{\prime}{}^{2})\cos\upsilon_{2}+\rho(k_{2}%
(1-\rho^{\prime}{}^{2})\sin\upsilon_{2}+\rho^{\prime}\rho^{\prime\prime}%
\cos\upsilon_{2})\right)  \right)  F_{2}\nonumber\\
&  +\left(  \frac{1}{\sqrt{1-\rho^{\prime}{}^{2}}}\left(
\begin{array}
[c]{l}%
-\rho^{\prime}(1-\rho^{\prime}{}^{2})\sin\upsilon_{2}\sin\upsilon_{3}\\
+\rho\left(
\begin{array}
[c]{l}%
-k_{2}(1-\rho^{\prime}{}^{2})\cos\upsilon_{2}\sin\upsilon_{3}\\
-k_{3}(1-\rho^{\prime}{}^{2})\cos\upsilon_{3}+\rho^{\prime}\rho^{\prime\prime
}\sin\upsilon_{2}\sin\upsilon_{3}%
\end{array}
\right)
\end{array}
\right)  \right)  F_{3}\nonumber\\
&  +\left(  \frac{1}{\sqrt{1-\rho^{\prime}{}^{2}}}\left(
\begin{array}
[c]{l}%
\rho^{\prime}(1-\rho^{\prime}{}^{2})\cos\upsilon_{3}\\
+\rho(-k_{3}(1-\rho^{\prime}{}^{2})\sin\upsilon_{2}\sin\upsilon_{3}%
-\rho^{\prime}\rho^{\prime\prime}\cos\upsilon_{3})
\end{array}
\right)  \right)  F_{4}, \label{17}%
\end{align}%
\begin{equation}
\emph{C}_{\upsilon_{2}\upsilon_{2}}=-\rho\sqrt{1-\rho^{\prime}{}^{2}}%
\cos\upsilon_{2}\cos\upsilon_{3}F_{2}-\rho\sqrt{1-\rho^{\prime}{}^{2}}%
\sin\upsilon_{2}\cos\upsilon_{3}F_{3}, \label{18}%
\end{equation}%
\begin{equation}
\emph{C}_{\upsilon_{2}\upsilon_{3}}=\emph{C}_{\upsilon_{3}\upsilon_{2}}%
=\rho\sqrt{1-\rho^{\prime}{}^{2}}\sin\upsilon_{2}\sin\upsilon_{3}F_{2}%
-\rho\sqrt{1-\rho^{\prime}{}^{2}}\cos\upsilon_{2}\sin\upsilon_{3}F_{3}
\label{19}%
\end{equation}
and
\begin{equation}
\emph{C}_{\upsilon_{3}\upsilon_{3}}=-\rho\sqrt{1-\rho^{\prime}{}^{2}}%
\cos\upsilon_{2}\cos\upsilon_{3}F_{2}-\rho\sqrt{1-\rho^{\prime}{}^{2}}%
\sin\upsilon_{2}\cos\upsilon_{3}F_{3}-\rho\sqrt{1-\rho^{\prime}{}^{2}}%
\sin\upsilon_{3}F_{4}, \label{20}%
\end{equation}
where%
\[%
\begin{array}
[c]{l}%
\emph{C}_{\upsilon_{1}\upsilon_{1}}^{\text{ }1}=\frac{1}{\sqrt{1-\rho^{\prime
}{}^{2}}}\left(
\begin{array}
[c]{l}%
\rho^{\prime}(-2k_{1}(1-\rho^{\prime}{}^{2})\cos\upsilon_{2}\cos\upsilon
_{3}-3\rho^{\prime\prime}\sqrt{1-\rho^{\prime}{}^{2}})\\
+\rho\left(
\begin{array}
[c]{l}%
(k_{1})^{2}\rho^{\prime}\sqrt{1-\rho^{\prime}{}^{2}}-k_{1}^{\prime}%
(1-\rho^{\prime}{}^{2})\cos\upsilon_{2}\cos\upsilon_{3}\\
+k_{1}(k_{2}(1-\rho^{\prime}{}^{2})\sin\upsilon_{2}+2\rho^{\prime}\rho
^{\prime\prime}\cos\upsilon_{2})\cos\upsilon_{3}-\rho^{\prime\prime\prime
}\sqrt{1-\rho^{\prime}{}^{2}}%
\end{array}
\right)
\end{array}
\right)  ,\\
\\
\emph{C}_{\upsilon_{1}\upsilon_{1}}^{\text{ }2}=\frac{-1}{(1-\rho^{\prime}%
{}^{2})^{\frac{3}{2}}}\left(
\begin{array}
[c]{l}%
-(1-\rho^{\prime}{}^{2})\left(
\begin{array}
[c]{l}%
k_{1}(1-2\rho^{\prime2})\sqrt{1-\rho^{\prime}{}^{2}}-2k_{2}\rho^{\prime
}(1-\rho^{\prime}{}^{2})\sin\upsilon_{2}\cos\upsilon_{3}\\
+\rho^{\prime\prime}(1-3\rho^{\prime}{}^{2})\cos\upsilon_{2}\cos\upsilon_{3}%
\end{array}
\right) \\
+\rho\left(
\begin{array}
[c]{l}%
(1-\rho^{\prime}{}^{2})^{2}((k_{1})^{2}+(k_{2})^{2})\cos\upsilon_{2}%
\cos\upsilon_{3}\\
+k_{1}^{\prime}\rho^{\prime}(1-\rho^{\prime}{}^{2})^{\frac{3}{2}}+k_{2}%
{}^{\prime}(1-\rho^{\prime}{}^{2})^{2}\sin\upsilon_{2}\cos\upsilon_{3}\\
+2k_{1}\rho^{\prime\prime}(1-\rho^{\prime}{}^{2})^{\frac{3}{2}}+(\rho
^{\prime\prime2}+\rho^{\prime}\rho^{\prime\prime\prime}(1-\rho^{\prime}{}%
^{2}))\cos\upsilon_{2}\cos\upsilon_{3}\\
+k_{2}(1-\rho^{\prime}{}^{2})(-k_{3}(1-\rho^{\prime}{}^{2})\sin\upsilon
_{3}-2\rho^{\prime}\rho^{\prime\prime}\sin\upsilon_{2}\cos\upsilon_{3})
\end{array}
\right)
\end{array}
\right)  ,
\end{array}
\]%
\[%
\begin{array}
[c]{l}%
\emph{C}_{\upsilon_{1}\upsilon_{1}}^{\text{ }3}=\frac{1}{(1-\rho^{\prime}%
{}^{2})^{\frac{3}{2}}}\left(
\begin{array}
[c]{l}%
(1-\rho^{\prime}{}^{2})\left(
\begin{array}
[c]{l}%
2\rho^{\prime}(1-\rho^{\prime}{}^{2})(k_{2}\cos\upsilon_{2}\cos\upsilon
_{3}-k_{3}\sin\upsilon_{3})\\
+\rho^{\prime\prime}(1-3\rho^{\prime}{}^{2})\sin\upsilon_{2}\cos\upsilon_{3}%
\end{array}
\right) \\
-\rho\left(
\begin{array}
[c]{l}%
(1-\rho^{\prime}{}^{2})^{2}((k_{2}^{2}+k_{3}^{2})\sin\upsilon_{2}%
-k_{2}^{\prime}\cos\upsilon_{2})\cos\upsilon_{3}\\
+k_{3}^{\prime}(1-\rho^{\prime}{}^{2})^{2}\sin v_{3}+2k_{3}\rho^{\prime}%
\rho^{\prime\prime}(-1+\rho^{\prime}{}^{2})\sin v_{3}\\
-k_{2}\rho^{\prime}(-1+\rho^{\prime}{}^{2})\left(  k_{1}\sqrt{1-\rho^{\prime
2}}+2\rho^{\prime\prime}\cos v_{2}\cos v_{3}\right) \\
+\sin v_{2}\cos v_{3}\left(  \rho^{\prime}\rho^{\prime\prime\prime}%
(1-\rho^{\prime2})+\rho^{\prime\prime2}\right)
\end{array}
\right)
\end{array}
\right)  ,\\
\\
\emph{C}_{\upsilon_{1}\upsilon_{1}}^{\text{ }4}=\frac{1}{(1-\rho^{\prime}%
{}^{2})^{\frac{3}{2}}}\left(
\begin{array}
[c]{l}%
(1-\rho^{\prime}{}^{2})(2\rho^{\prime}k_{3}(1-\rho^{\prime}{}^{2})\sin
\upsilon_{2}\cos\upsilon_{3}+\rho^{\prime\prime}(1-3\rho^{\prime}{}^{2}%
)\sin\upsilon_{3})\\
-\rho^{\prime}\left(
\begin{array}
[c]{l}%
k_{3}(1-\rho^{\prime2})^{2}(-k_{2}\cos\upsilon_{2}\cos\upsilon_{3}+k_{3}%
\sin\upsilon_{3})\\
-(1-\rho^{\prime2})((1-\rho^{\prime2})k_{3}^{\prime}\sin\upsilon_{2}%
\cos\upsilon_{3}+2k_{3}\rho^{\prime}\rho^{\prime\prime})\\
+(\rho^{\prime\prime2}-\rho^{\prime}\rho^{\prime\prime\prime}(1-\rho^{\prime
2}))\sin\upsilon_{3}%
\end{array}
\right)
\end{array}
\right)  .
\end{array}
\]
Thus, from (\ref{6y}), (\ref{12}) and (\ref{15})-(\ref{20}), the coefficients
of the second fundamental form are given by%
\begin{equation}
\left.
\begin{array}
[c]{l}%
h_{11}=\frac{\rho}{\rho^{\prime2}-1}\left(
\begin{array}
[c]{l}%
\left(  (k_{2})^{2}\cos^{2}v_{3}-k_{2}k_{3}\cos v_{2}\sin(2v_{3})+(k_{3}%
)^{2}\left(  \cos^{2}v_{3}\sin^{2}v_{2}+\sin^{2}v_{3}\right)  \right)
(1-\rho^{\prime}{}^{2})^{2}\\
+(k_{1})^{2}(1-\rho^{\prime}{}^{2})\left(  (1-\rho^{\prime}{}^{2})\cos
^{2}v_{2}\cos^{2}v_{3}+\rho^{\prime}{}^{2}\right)  +\rho^{\prime\prime2}\\
+2k_{1}\sqrt{1-\rho^{\prime}{}^{2}}(k_{2}\rho^{\prime}(1-\rho^{\prime}{}%
^{2})\sin v_{2}+\rho^{\prime\prime}\cos v_{2})\cos v_{3}%
\end{array}
\right) \\
\text{ \ \ \ \ \ }+k_{1}\sqrt{1-\rho^{\prime2}}\cos\upsilon_{2}\cos
\upsilon_{3}+\rho^{\prime\prime},\\
\\
h_{12}=h_{21}=\rho\left(  -k_{1}\rho^{\prime}\sqrt{1-\rho^{\prime}{}^{2}}%
\sin\upsilon_{2}+(1-\rho^{\prime}{}^{2})(k_{3}\cos\upsilon_{2}\sin\upsilon
_{3}-k_{2}\cos\upsilon_{3})\right)  \cos\upsilon_{3},\\
\\
h_{13}=h_{31}=\rho\left(  -k_{1}\rho^{\prime}\sqrt{1-\rho^{\prime}{}^{2}}%
\cos\upsilon_{2}\sin\upsilon_{3}-k_{3}(1-\rho^{\prime}{}^{2})\sin\upsilon
_{2}\right)  ,\\
\\
h_{22}=-\rho(1-\rho^{\prime2})\cos^{2}\upsilon_{3},\\
\\
h_{23}=h_{32}=0,\\
\\
h_{33}=-\rho(1-\rho^{\prime2})
\end{array}
\right\}  \label{21}%
\end{equation}
and it implies%
\begin{equation}
\det[h_{ij}]=\rho^{2}(1-\rho^{\prime2})\left(
\begin{array}
[c]{l}%
(1-\rho^{\prime2})(k_{1}\sqrt{1-\rho^{\prime}{}^{2}}\cos\upsilon_{2}%
\cos\upsilon_{3}+\rho^{\prime\prime})\\
-\rho\left(
\begin{array}
[c]{l}%
(k_{1})^{2}(1-\rho^{\prime2})\cos^{2}\upsilon_{2}\cos^{2}\upsilon_{3}\\
+2k_{1}\rho^{\prime\prime}\sqrt{1-\rho^{\prime}{}^{2}}\cos\upsilon_{2}%
\cos\upsilon_{3}+\rho^{\prime\prime2}%
\end{array}
\right)
\end{array}
\right)  \cos^{2}\upsilon_{3}. \label{22}%
\end{equation}
So, from (\ref{yy4}), (\ref{14}) and (\ref{22}), we have

\begin{theorem}
The Gaussian curvature of the canal hypersurface (\ref{4-canal}) in Euclidean
4-space is%
\begin{equation}
K=\frac{\left(  (1-\rho^{\prime}{}^{2})(k_{1}\sqrt{1-\rho^{\prime}{}^{2}}%
\cos\upsilon_{2}\cos\upsilon_{3}+\rho^{\prime\prime})-\rho\left(
\begin{array}
[c]{l}%
(k_{1})^{2}(1-\rho^{\prime2})\cos^{2}\upsilon_{2}\cos^{2}\upsilon_{3}%
+\rho^{\prime\prime2}\\
+2k_{1}\rho^{\prime\prime}\sqrt{1-\rho^{\prime}{}^{2}}\cos\upsilon_{2}%
\cos\upsilon_{3}%
\end{array}
\right)  \right)  }{\rho^{2}\left(  \rho\left(  k_{1}\sqrt{1-\rho^{\prime}%
{}^{2}}\cos\upsilon_{2}\cos\upsilon_{3}+\rho^{\prime\prime}\right)
-1+\rho^{\prime}{}^{2}\right)  ^{2}}. \label{gauss}%
\end{equation}

\end{theorem}

\begin{theorem}
The canal hypersurface (\ref{4-canal}) in Euclidean 4-space is flat if and
only if it is a circular hypercylinder or circular hypercone.
\end{theorem}

\begin{proof}
If the canal hypersurface (\ref{4-canal}) in Euclidean 4-space is flat, then
from (\ref{gauss}) it must be%
\begin{equation}
k_{1}\sqrt{1-\rho^{\prime}{}^{2}}(1-\rho^{\prime}{}^{2}-2\rho\rho
^{\prime\prime})\cos\upsilon_{2}\cos\upsilon_{3}-\rho(k_{1})^{2}%
(1-\rho^{\prime2})\cos^{2}\upsilon_{2}\cos^{2}\upsilon_{3}+(1-\rho^{\prime}%
{}^{2}-\rho\rho^{\prime\prime})\rho^{\prime\prime}=0. \label{p1}%
\end{equation}
Since the set $\{1,\cos v_{2}\cos v_{3},\cos^{2}\upsilon_{2}\cos^{2}%
\upsilon_{3}\}$ is linearly independent, we have%
\begin{equation}
\left.
\begin{array}
[c]{l}%
k_{1}\sqrt{1-\rho^{\prime}{}^{2}}(1-\rho^{\prime}{}^{2}-2\rho\rho
^{\prime\prime})=0,\\
\rho(k_{1})^{2}(1-\rho^{\prime2})=0,\\
(1-\rho^{\prime}{}^{2}-\rho\rho^{\prime\prime})\rho^{\prime\prime}=0.
\end{array}
\right\}  \label{p2}%
\end{equation}
From the second equation of (\ref{p2}), since $\rho\neq0$ and $1-\rho
^{\prime2}\neq0,$ we have $k_{1}=0$ and so, (\ref{4-canal}) is a hypersurface
of revolution. Also, in this case, the first equation of (\ref{p2}) holds,
too. If we use $k_{1}=0$ in (\ref{gauss}), we have%
\begin{equation}
K=\frac{\rho^{\prime\prime}}{\rho^{2}(-1+\rho^{\prime}{}^{2}+\rho\rho
^{\prime\prime})}. \label{p3}%
\end{equation}
Since the hypersurface is flat, (\ref{p3}) implies $\rho^{\prime\prime}=0$,
that is, $\rho(v_{1})=av_{1}+b,$ $a,b\in%
%TCIMACRO{\U{211d} }%
%BeginExpansion
\mathbb{R}
%EndExpansion
,$ $a\neq\pm1$. From this, the third equation of (\ref{p2}) holds, too.
Therefore, (\ref{4-canal}) is a circular hypercylinder when $a=0$, or a
circular hypercone when $a\neq0,$ $a\neq\pm1$.

Conversely, if $k_{1}=0$ and $\rho(v_{1})=av_{1}+b$ (i.e., if (\ref{4-canal})
is a circular hypercylinder or a circular hypercone), then we have $K=0$ and
this completes the proof.
\end{proof}

Also, after finding the inverse of the matrix of the first fundamental form
and using this and (\ref{21}) in (\ref{7y}), the shape operator of the canal
hypersurface (\ref{4-canal}) is obtained by%
\begin{equation}
S=\left[
\begin{array}
[c]{ccc}%
S_{11} & S_{12} & S_{13}\\
S_{21} & S_{22} & S_{23}\\
S_{31} & S_{32} & S_{33}%
\end{array}
\right]  , \label{23}%
\end{equation}
where%
\[
\left.
\begin{array}
[c]{l}%
S_{11}=\frac{1}{Q^{2}}\left(
\begin{array}
[c]{l}%
(1-\rho^{\prime}{}^{2})\left(  k_{1}\sqrt{1-\rho^{\prime}{}^{2}}\cos
\upsilon_{2}\cos\upsilon_{3}+\rho^{\prime\prime}\right) \\
-\rho\left(  (k_{1})^{2}(1-\rho^{\prime}{}^{2})\cos^{2}\upsilon_{2}\cos
^{2}\upsilon_{3}+2k_{1}\rho^{\prime\prime}\sqrt{1-\rho^{\prime}{}^{2}}\cos
v_{2}\cos v_{3}+\rho^{\prime\prime2}\right)
\end{array}
\right)  ,\\
\\
S_{21}=\frac{1}{\rho Q}\left(  k_{1}\rho^{\prime}\sqrt{1-\rho^{\prime}{}^{2}%
}\sin v_{2}\sec v_{3}+k_{2}(1-\rho^{\prime}{}^{2})-k_{3}(1-\rho^{\prime}{}%
^{2})\cos v_{2}\tan v_{3}\right)  ,\\
\\
S_{31}=\frac{1}{\rho Q^{2}}\left(
\begin{array}
[c]{l}%
Qk_{3}(1-\rho^{\prime}{}^{2})\sin v_{2}\\
+k_{1}\rho^{\prime}\cos v_{2}\sin v_{3}\left(  -(1-\rho^{\prime}{}^{2}%
)^{\frac{3}{2}}+\rho\left(
\begin{array}
[c]{l}%
k_{1}(1-\rho^{\prime}{}^{2})\cos v_{2}\cos v_{3}\\
+\sqrt{1-\rho^{\prime}{}^{2}}\rho^{\prime\prime}%
\end{array}
\right)  \right)
\end{array}
\right)  ,\\
\\
S_{22}=S_{33}=-\frac{1}{\rho},\\
\\
S_{12}=S_{13}=S_{23}=S_{32}=0.
\end{array}
\right\}
\]

Hence from (\ref{yy5}) and (\ref{23}), we get

\begin{theorem}
The mean curvature of the canal hypersurface (\ref{4-canal}) in Euclidean
4-space is
\begin{equation}
H=\frac{\left(
\begin{array}
[c]{l}%
-3\rho^{2}\left(  (k_{1})^{2}(1-\rho^{\prime}{}^{2})\cos^{2}\upsilon_{2}%
\cos^{2}\upsilon_{3}+2k_{1}\rho^{\prime\prime}\sqrt{1-\rho^{\prime}{}^{2}}%
\cos\upsilon_{2}\cos\upsilon_{3}+\rho^{\prime\prime2}\right) \\
-2(1-\rho^{\prime}{}^{2})^{2}+5\rho(1-\rho^{\prime}{}^{2})(k_{1}\sqrt
{1-\rho^{\prime}{}^{2}}\cos\upsilon_{2}\cos\upsilon_{3}+\rho^{\prime\prime})
\end{array}
\right)  }{3\rho\left(  \rho\left(  k_{1}\sqrt{1-\rho^{\prime}{}^{2}}%
\cos\upsilon_{2}\cos\upsilon_{3}+\rho^{\prime\prime}\right)  -1+\rho^{\prime
}{}^{2}\right)  ^{2}}. \label{mean}%
\end{equation}

\end{theorem}

\begin{theorem}
The canal hypersurface (\ref{4-canal}) in Euclidean 4-space is minimal if and
only if it is a hypersurface of revolution parametrized by%
\begin{equation}
\emph{R}(\upsilon_{1},\upsilon_{2},\upsilon_{3})=(\upsilon_{1}-\rho
\rho^{\prime},\pm\rho\sqrt{1-\rho^{\prime}{}^{2}}\cos\upsilon_{2}\cos
\upsilon_{3},\pm\rho\sqrt{1-\rho^{\prime}{}^{2}}\sin\upsilon_{2}\cos
\upsilon_{3},\pm\rho\sqrt{1-\rho^{\prime}{}^{2}}\sin\upsilon_{3}), \label{sr}%
\end{equation}
where $\rho(\upsilon_{1})$ is given by $%
%TCIMACRO{\dint }%
%BeginExpansion
{\displaystyle\int}
%EndExpansion
\frac{d\rho}{\sqrt{1-\left(  \frac{a}{\rho}\right)  ^{\frac{4}{3}}}}=\pm
v_{1}+b,$ $a,b\in%
%TCIMACRO{\U{211d} }%
%BeginExpansion
\mathbb{R}
%EndExpansion
$.
\end{theorem}

\begin{proof}
If the canal hypersurface (\ref{4-canal}) in Euclidean 4-space is minimal,
then from (\ref{mean}) it must be%
\begin{align}
&  k_{1}\sqrt{1-\rho^{\prime}{}^{2}}(5\rho(1-\rho^{\prime}{}^{2})-6\rho
^{2}\rho^{\prime\prime})\cos\upsilon_{2}\cos\upsilon_{3}-3\rho^{2}(k_{1}%
)^{2}(1-\rho^{\prime}{}^{2})\cos^{2}\upsilon_{2}\cos^{2}\upsilon
_{3}\nonumber\\
&  \text{ }+(2-2\rho^{\prime}{}^{2}-3\rho\rho^{\prime\prime})(-1+\rho^{\prime
}{}^{2}+\rho\rho^{\prime\prime})=0. \label{p4}%
\end{align}
Since the set $\{1,\cos v_{2}\cos v_{3},\cos^{2}\upsilon_{2}\cos^{2}%
\upsilon_{3}\}$ is linearly independent, we have%
\begin{equation}
\left.
\begin{array}
[c]{l}%
k_{1}\sqrt{1-\rho^{\prime}{}^{2}}(5\rho(1-\rho^{\prime}{}^{2})-6\rho^{2}%
\rho^{\prime\prime})=0,\\
3\rho^{2}(k_{1})^{2}(1-\rho^{\prime}{}^{2})=0,\\
(2-2\rho^{\prime}{}^{2}-3\rho\rho^{\prime\prime})(-1+\rho^{\prime}{}^{2}%
+\rho\rho^{\prime\prime})=0.
\end{array}
\right\}  \label{p5}%
\end{equation}
From the second equation of (\ref{p5}), we have $k_{1}=0$ and then the first
equation of (\ref{p5}) holds, too. If we use $k_{1}=0$ in (\ref{mean}), we
have%
\begin{equation}
H=\frac{2-2\rho^{\prime}{}^{2}-3\rho\rho^{\prime\prime}}{3\rho(-1+\rho
^{\prime}{}^{2}+\rho\rho^{\prime\prime})}. \label{p6}%
\end{equation}
So, if the canal hypersurface (\ref{4-canal}) in Euclidean 4-space is minimal,
then from (\ref{p6}) $\rho(v_{1})$ must satisfy the differential equation
\begin{equation}
2-2\rho^{\prime}{}(v_{1})^{2}-3\rho(v_{1})\rho^{\prime\prime}(v_{1})=0.
\label{p7}%
\end{equation}
Now, let us solve (\ref{p7}).

If we take $\rho^{\prime}{}(v_{1})=f(v_{1}),$ we get
\begin{equation}
\rho^{\prime}{}^{\prime}=f^{\prime}=\frac{df}{d\rho}\frac{d\rho}{dv_{1}}%
=\frac{df}{d\rho}f. \label{p8}%
\end{equation}
Using (\ref{p8}) in (\ref{p7}), we have%
\begin{equation}
3\rho\frac{df}{d\rho}f+2f^{2}-2=0. \label{p9}%
\end{equation}
From (\ref{p7}), $\rho^{\prime}{}(v_{1})=f(v_{1})\neq0$ and so we reach that%
\begin{equation}
\frac{3f}{2(1-f^{2})}df=\frac{d\rho}{\rho}. \label{p10}%
\end{equation}
By integrating (\ref{p10}), we have%
\begin{equation}
\text{ }\text{ }f=\pm\sqrt{1-\left(  \frac{a}{\rho}\right)  ^{\frac{4}{3}}},
\label{p11}%
\end{equation}
where $a$ is constant. Since $\rho^{\prime}{}=\frac{d\rho}{dv_{1}}=f,$ from
(\ref{p11}) we get%
\begin{equation}%
%TCIMACRO{\dint }%
%BeginExpansion
{\displaystyle\int}
%EndExpansion
\frac{d\rho}{\sqrt{1-\left(  \frac{a}{\rho}\right)  ^{\frac{4}{3}}}}=\pm%
%TCIMACRO{\dint }%
%BeginExpansion
{\displaystyle\int}
%EndExpansion
dv_{1}. \label{p12y}%
\end{equation}
Since $k_{1}=0$, without loss of generality, we can suppose the curve
$\alpha(\upsilon_{1})$ as $\alpha(\upsilon_{1})=(\upsilon_{1},0,0,0)$ and
$F_{1}=(1,0,0,0),$ $F_{2}=(0,1,0,0),$ $F_{3}=(0,0,1,0),$ $F_{4}=(0,0,0,1)$.
Then, (\ref{4-canal}) can be parametrized by (\ref{sr}) and from (\ref{p12y}),
$\rho(v_{1})$ satisfies $%
%TCIMACRO{\dint }%
%BeginExpansion
{\displaystyle\int}
%EndExpansion
\frac{d\rho}{\sqrt{1-\left(  \frac{a}{\rho}\right)  ^{\frac{4}{3}}}}=\pm
v_{1}+b,$ $a,b\in%
%TCIMACRO{\U{211d} }%
%BeginExpansion
\mathbb{R}
%EndExpansion
.$

Conversely, if (\ref{4-canal}) is parametrized by (\ref{sr}), where
$\rho(v_{1})$ satisfies $%
%TCIMACRO{\dint }%
%BeginExpansion
{\displaystyle\int}
%EndExpansion
\frac{d\rho}{\sqrt{1-\left(  \frac{a}{\rho}\right)  ^{\frac{4}{3}}}}=\pm
v_{1}+b,$ $a,b\in%
%TCIMACRO{\U{211d} }%
%BeginExpansion
\mathbb{R}
%EndExpansion
$, then we have $H=0$ and this completes the proof.
\end{proof}

Also, we know that \cite{Pinl}, the only minimal hypersurface of revolution
(except the hyperplane) in Euclidean space is the generalized catenoid. Thus,
from the last Theorem, we have

\begin{corollary}
The canal hypersurface (\ref{4-canal}) is minimal if and only if it is a
generalized catenoid.
\end{corollary}

Here, from (\ref{gauss}) and (\ref{mean}), we can state the following theorem
which gives an important relation between Gaussian and mean curvatures:

\begin{theorem}
The Gaussian curvature $K$ and the mean curvature $H$ of canal hypersurfaces
(\ref{4-canal})\ in Euclidean 4-space satisfy%
\begin{equation}
H=\frac{1}{3}(K\rho^{2}-\frac{2}{\rho}).\label{24}%
\end{equation}

\end{theorem}

Now, if%
\begin{equation}
H_{\upsilon_{i}}K_{\upsilon_{j}}-H_{\upsilon_{j}}K_{\upsilon_{i}}=0,\text{
}i\neq j,\text{ }i,j=1,2,3, \label{w1}%
\end{equation}
holds on a hypersurface, then we call the hypersurface as $(H,K)_{\{ij\}}%
$-Weingarten hypersurface, where $X_{\upsilon_{i}}=\frac{\partial X}{\partial
v_{i}}$. So, from (\ref{gauss}) and (\ref{mean}) we have

\begin{theorem}
The canal hypersurface (\ref{4-canal}) in Euclidean 4-space is $(H,K)_{\{23\}}%
$-Weingarten hypersurface.
\end{theorem}

Also, from (\ref{23}) we have%

\begin{equation}
\det(S-\kappa I_{3})=-\frac{(\kappa\rho+1)^{2}}{\rho^{2}Q{}^{2}}\left(
\begin{array}
[c]{l}%
\kappa\rho^{2}\left(
\begin{array}
[c]{l}%
2k_{1}\rho^{\prime\prime}\sqrt{1-\rho^{\prime2}}\cos\upsilon_{2}\cos
\upsilon_{3}\\
+(k_{1})^{2}\left(  1-\rho^{\prime}{}^{2}\right)  \cos^{2}\upsilon_{2}\cos
^{2}\upsilon_{3}+\rho^{\prime\prime2}%
\end{array}
\right) \\
+\rho\left(
\begin{array}
[c]{l}%
2k_{1}\sqrt{1-\rho^{\prime2}}\left(  \kappa\rho^{\prime2}-\kappa+\rho
^{\prime\prime}\right)  \cos\upsilon_{2}\cos\upsilon_{3}\\
+\rho^{\prime\prime}\left(  2\kappa\rho^{\prime2}-2\kappa+\rho^{\prime\prime
}\right) \\
+(k_{1})^{2}\left(  1-\rho^{\prime}{}^{2}\right)  \cos^{2}\upsilon_{2}\cos
^{2}\upsilon_{3}%
\end{array}
\right) \\
-\left(  1-\rho^{\prime}{}^{2}\right)  \left(  \kappa\rho^{\prime2}%
-\kappa+k_{1}\sqrt{1-\rho^{\prime2}}\cos\upsilon_{2}\cos\upsilon_{3}%
+\rho^{\prime\prime}\right)
\end{array}
\right)  . \label{29''}%
\end{equation}
By solving the equation $\det(S-\kappa I_{3})=0$ from (\ref{29''}), we obtain
the principal curvatures of the canal hypersurfaces (\ref{4-canal}) in $E^{4}$
as follows:

\begin{theorem}
The principal curvatures of the canal hypersurfaces (\ref{4-canal}) in $E^{4}$
are%
\begin{equation}
\kappa_{1}=\kappa_{2}=-\frac{1}{\rho}\text{ \ \ and \ \ }\kappa_{3}=K\rho^{2}.
\label{29'}%
\end{equation}

\end{theorem}

Here, we know that if $\rho(\upsilon_{1})=\lambda$ is a constant, then the
canal hypersurface is called tubular or pipe hypersurface and from
(\ref{n-canal}) the tubular hypersurface in $E^{n}$ can be given by%

\begin{equation}
X(\upsilon_{1},\upsilon_{2},\upsilon_{3},...,\upsilon_{n-1})=\alpha
(\upsilon_{1})\pm\lambda\left[
\begin{array}
[c]{l}%
\left(  \prod\limits_{k=2}^{n-1}\cos(x_{k})\right)  F_{2}(\upsilon_{1})\\
+%
%TCIMACRO{\dsum \limits_{i=3}^{n-1}}%
%BeginExpansion
{\displaystyle\sum\limits_{i=3}^{n-1}}
%EndExpansion
\left(  \sin(\upsilon_{n+1-i})\prod\limits_{k=n+2-i}^{n-1}\cos(x_{k})\right)
F_{i}(\upsilon_{1})\\
+\sin(\upsilon_{n-1})F_{n}(\upsilon_{1})
\end{array}
\right]  . \label{25}%
\end{equation}
So, from (\ref{25}) the tubular hypersurface in $E^{4}$ is%

\begin{equation}
\emph{T}(\upsilon_{1},\upsilon_{2},\upsilon_{3})=\alpha(\upsilon_{1}%
)\pm\lambda\left[  \cos\upsilon_{2}\cos\upsilon_{3}F_{2}(\upsilon_{1}%
)+\sin\upsilon_{2}\cos\upsilon_{3}F_{3}(\upsilon_{1})+\sin\upsilon_{3}%
F_{4}(\upsilon_{1})\right]  . \label{26}%
\end{equation}
Here, by taking "$\pm$" as "$+$" in (\ref{26}), we get

\begin{theorem}
The Gaussian and mean curvatures of the tubular hypersurface (\ref{26}) in
Euclidean 4-space are%
\begin{equation}
K=\frac{k_{1}\cos\upsilon_{2}\cos\upsilon_{3}}{\lambda^{2}(1-k_{1}\lambda
\cos\upsilon_{2}\cos\upsilon_{3})} \label{27}%
\end{equation}
and%
\begin{equation}
H=\frac{2-3k_{1}\lambda\cos\upsilon_{2}\cos\upsilon_{3}}{3\lambda
(-1+k_{1}\lambda\cos\upsilon_{2}\cos\upsilon_{3})}, \label{28}%
\end{equation}
respectively.
\end{theorem}

So, from (\ref{27}) and (\ref{28}) we get

\begin{theorem}
The tubular hypersurface (\ref{26}) in Euclidean 4-space is $(H,K)_{\{12\}}$
and $(H,K)_{\{13\}}$- Weingarten hypersurface.
\end{theorem}

Also, we know that, a hypersurface is called a linear Weingarten hypersurface,
if it satisfies
\begin{equation}
aH+bK=c, \label{31}%
\end{equation}
where $a,b,c$ are not all zero constants. Thus, from (\ref{24}), we have

\begin{theorem}
The tubular hypersurface (\ref{26}) in Euclidean 4-space is a linear
Weingarten hypersurface.
\end{theorem}

\begin{proof}
The equation (\ref{24}) on the tubular hypersurface (\ref{26}) implies
\begin{equation}
-3\lambda H+\lambda^{3}K=2. \label{32}%
\end{equation}
So, from the definition of a linear Weingarten hypersurface, the proof completes.
\end{proof}

\section{Conclusion and Future Work}

In this study, firstly we obtain the general expression of canal hypersurfaces
in Euclidean $n$-space and we deal with canal hypersurfaces in Euclidean
4-space with the aid of this expression. In this sense, we obtain the Gaussian
curvature and the mean curvature of canal hypersurfaces in $E^{4}$ and give an
important relation between the Gaussian and mean curvatures. Also by taking
the radius function as a constant, we state the tubular hypersurfaces in
Euclidean spaces and give some results about tubular hypersurfaces in $E^{4}$.
In this context, we prove that the tubular hypersurfaces are linear Weingarten
hypersurfaces in $E^{4}$.

We hope that, this study will give a new perspective to readers who deal with
the canal hypersurfaces in $E^{4}$ and $E^{n}$. As open problems, this study
can be handled in Minkowskian, Galilean and pseudo-Galilean 4-spaces in the
near future. Also, some important characterizations for the Laplace-Beltrami
operators on the canal hypersurfaces or different classifications for canal
hypersurfaces in different 4-dimensional spaces can be investigated.


\bigskip

\begin{thebibliography}{99}                                                                                               %


\bibitem {Aslan}S. Aslan and Y. Yayl\i; \textit{Canal Surfaces with
Quaternions}, Adv. Appl. Clifford Algebr., 26, (2016), 31-38.

\bibitem {Chen}B-Y. Chen and K. Yano; \textit{Special Conformally Flat Spaces
and Canal Hypersurfaces}, Tohoku Math. J., 25, (1973), 177-184.

\bibitem {Cheng}B. Cheng; \textit{Frenet Formulas in n-Dimensions and Some
Applications,} Pi Mu Epsilon Journal, 7(10), (1984), 629-635.

\bibitem {Garcia}R. Garcia, J. Llibre and J. Sotomayor;\textit{ Lines of
Principal Curvature on Canal Surfaces in }$%
%TCIMACRO{\U{211d} }%
%BeginExpansion
\mathbb{R}
%EndExpansion
^{3}$, An. Acad. Brasil. Cienc., 78(3), (2006), 405-415.

\bibitem {Gray}A. Gray; Modern Differential Geometry of Curves and Surfaces
with Mathematica, 2nd edn. CRC Press, Boca Raton, (1999).

\bibitem {Guler}E. G\"{u}ler, H.H. Hac\i saliho\u{g}lu and Y.H. Kim;
\textit{The Gauss map and the third Laplace-Beltrami operator of the
rotational hypersurface in 4-Space}, Symmetry, 10(9), (2018), 1-11.

\bibitem {Hartman}E. Hartman; Geometry and Algorithms for Computer Aided
Design, Dept. of Math. Darmstadt Univ. of Technology, 2003.

\bibitem {Izumiya}S. Izumiya and M .Takahashi; \textit{On caustics of
submanifolds and canal hypersurfaces in Euclidean space, }Topology Appl., 159,
(2012), 501-508.

\bibitem {Karacan}M.K. Karacan, H. Es and Y. Yayl\i; \textit{Singular Points
of Tubular Surfaces in Minkowski 3-Space}, Sarajevo J. Math., 2(14), (2006), 73-82.

\bibitem {Karacan2}M.K. Karacan and Y. Tuncer; \textit{Tubular Surfaces of
Weingarten Types in Galilean and Pseudo-Galilean}, Bull. Math. Anal. Appl.,
5(2), (2013), 87-100.

\bibitem {Karacan3}M.K. Karacan, D.W. Yoon and Y. Tuncer; \textit{Tubular
Surfaces of Weingarten Types in Minkowski 3-Space, }Gen. Math. Notes, 22(1),
(2014), 44-56.

\bibitem {Kim}Y.H. Kim, H. Liu and J. Qian; \textit{Some Characterizations of
Canal Surfaces}, Bull. Korean Math. Soc., 53(2), (2016), 461-477.

\bibitem {Krivos}S.N. Krivoshapko and C.A.B. Hyeng; \textit{Classification of
Cyclic Surfaces and Geometrical Research of Canal Surfaces}, International
Journal of Research and Reviews in Applied Sciences, 12(3), (2012), 360-374.

\bibitem {Kucuk}Z. K\"{u}\c{c}\"{u}karslan Y\"{u}zba\c{s}\i\ and D.W. Yoon;
\textit{Tubular Surfaces with Galilean Darboux Frame in }$G_{3}$, Journal of
Mathematical Physics, Analysis, Geometry, 15(2), (2019), 278--287.

\bibitem {Maekawa}T. Maekawa, N.M. Patrikalakis, T. Sakkalis and G. Yu;
\textit{Analysis and Applications of Pipe Surfaces}, Comput. Aided Geom.
Design, 15, (1998), 437-458.

\bibitem {Peter}M. Peternell and H. Pottmann; \textit{Computing Rational
Parametrizations of Canal Surfaces}, J. Symbolic Comput., 23, (1997), 255-266.

\bibitem {Pinl}M. Pinl and W. Ziller; \textit{Minimal hypersurfaces in spaces
of constant curvature}, Journal of Differential Geometry, 11, (1976), 335-343.

\bibitem {Ro}J.S. Ro and D.W. Yoon; \textit{Tubes of Weingarten Type in a
Euclidean 3-Space}, Journal of the Chungcheong Mathematical Society, 22(3),
(2009), 359-366.

\bibitem {Ucum}A. U\c{c}um and K. \.{I}larslan; \textit{New Types of Canal
Surfaces in Minkowski 3-Space, }Adv. Appl. Clifford Algebr., 26, (2016), 449-468.

\bibitem {Xu}Z. Xu, R. Feng and J-G. Sun; \textit{Analytic and Algebraic
Properties of Canal Surfaces}, J. Comput. Appl. Math., 195, (2006), 220-228.
\end{thebibliography}
\end{document}